\documentclass{amsart}
\usepackage{amssymb,latexsym}
\theoremstyle{plain}
\newtheorem{theorem}{Theorem}
\newtheorem{corollary}{Corollary}

\newtheorem{lemma}{Lemma}

\newtheorem{proposition}{Proposition}

\theoremstyle{definition}

\theoremstyle{remark}

\numberwithin{equation}{section}
\newcommand{\R}{\mathbb R}

\newcommand{\Z}{\mathbb Z}
\newcommand{\C}{\mathbb C}

\newcommand{\Rn}{\mathbb R^n}

\begin{document}

\title[Uncertainty Principle, \dots]{Uncertainty Principle of Morgan type and
Schr\"odinger Evolutions}
%%%%%%%%%%%%%%%%%%
%Author information
%%%%%%%%%%%%%%%%%%%%
\author{L. Escauriaza}
\address[L. Escauriaza]{UPV/EHU\\Dpto. de Matem\'aticas\\Apto. 644, 48080
Bilbao, Spain.}
\email{luis.escauriaza@ehu.es}
\thanks{The first and fourth authors are supported  by MEC grant,
MTM2004-03029, the second and third authors by NSF grants DMS-0456583 and
DMS-0456833 respectively}
%\thanks{}
%%%%%%%%%%%%%%%%%%%%%%
\author{C. E. Kenig}
\address[C. E. Kenig]{Department of Mathematics\\University of
Chicago\\Chicago, Il. 60637 \\USA.}
\email{cek@math.uchicago.edu}
%\thanks{}
%%%%%%%%%%%%%%%%%%%%%%
\author{G. Ponce}
\address[G. Ponce]{Department of Mathematics\\
University of California\\
Santa Barbara, CA 93106\\
USA.}
\email{ponce@math.ucsb.edu}
%\thanks{}
%%%%%%%%%%%%%%%%%%%%%%
\author{L. Vega}
\address[L. Vega]{UPV/EHU\\Dpto. de Matem\'aticas\\Apto. 644, 48080
Bilbao, Spain.}
\email{luis.vega@ehu.es}
%\thanks{}
%%%%%%%%%%%%%%%%%%%%%%
%\date{}
%\dedicatory{}
%%%%%%%%%%%%%%
\begin{abstract}
We prove unique continuation properties for solutions of evolution Schr\"odinger equation with 
time dependent potentials. In the case of the free solution these correspond to uncertainly principles referred to as being of Morgan type. As  an application of our method we also obtain results concerning the possible  concentration profiles of  solutions of semi-linear Schr\"odinger equations. 
\end{abstract}
\maketitle
%%%%%%%%%%%%%%%%%%
\begin{section}{Introduction}\label{S: Introduction}
In this paper we continue our study initiated in 
\cite{ekpv08} \cite{ekpv08b}, and \cite{ekpv09} on unique continuation properties of
solutions of Schr\"odinger equations of the form
\begin{equation}
\label{E:1.1}
i \partial_t u + \triangle u = V(x,t) u,\;\;\;\;\;\; (x,t)\in \R^n\times
[0,1].
\end{equation}

The goal is to obtain sufficient conditions on the behavior of the
solution $u$
at two different times and on the potential $V$ which guarantee that
$u\equiv 0$ in $\Rn\times [0,1]$.
Under appropriate assumptions this result will allow us to extend these
conditions to the difference $v=u_1-u_2$ of two solutions $u_1,\;u_2$ of
semi-linear
Schr\"odinger equation
\begin{equation}
\label{E: NLS}
i \partial_t u + \triangle u = F(u,\overline u),
\end{equation}
from which one can infer that $u_1\equiv u_2$, see \cite{ekpv06}.
\vskip.1in

Defining the Fourier transform of a function $\,f\,$ as
$$
\widehat f(\xi)=(2\pi)^{-n/2}\;\int_{\Rn} \,e^{-i\,\xi\cdot x}
\,f(x)\,dx,
$$
the identity
\begin{equation}
\label{formula-l}
\begin{aligned}
&e^{it\Delta}u_0(x)=u(x,t)\\
\\
&= (4\pi it)^{-\frac n 2} \int_{\Rn}e^{\frac{ i |x-y|^2}{4t}}u_0(y)\, dy=
\left(2\pi it\right)^{-\frac
n2}e^{\frac{i|x|^2}{4t}}\widehat{e^{\frac{i|\,\cdot\,|^2}{4t}}u_0}\left(\frac
x{2t}\right),
\end{aligned}
\end{equation}
tells us that this kind of results for the free solution of the
Schr\"odinger equation with data $u_0$
$$
i\partial_tu+\triangle
u=0,\;\;\;\;\;\;u(x,0)=u_0(x),\;\;\;\;(x,t)\in\Rn\times \R,
$$
is related to uncertainty principles. In this regard, one has the well
known result of G. H. Hardy \cite{hardy} for $n=1$ and its extension to higher dimensions $n\geq 2$ established in \cite{SisuT} :
\begin{equation}
\label{hardy}
\text{If} \;\;\;f(x)=O(e^{-\frac{|x|^2}{\beta^2}}),\;\;\widehat
f(\xi)=O(e^{-\frac{4|\xi|^2}{\alpha^2}}),\;\;\text{and}\;\;\alpha\,\beta<4,\;\text{then}\;f\equiv
0.
\end{equation}
Moreover, if $\,\alpha\,\beta =4$, then $f(x)=c\,e^{-\frac{|x|^2}{\beta^2}}$.

Using \eqref{formula-l},  \eqref{hardy} can be rewritten in terms of the free solution
of the Schr\"odinger equation :
$$
\text{If} \;\;\;u_0(x)=O(e^{-\frac{|x|^2}{\beta^2}}),\;\;
e^{it\Delta}u_0(x)=O(e^{-\frac{|x|^2}{\alpha^2}}),\;\;\text{and}\;\;\alpha\,\beta<4t,\;\text{then}\;u_0\equiv
0.
$$
Also, if $\alpha\,\beta=4t$, then $u_0(x)=c
\,e^{-i\frac{|x|^2}{4t}}\,e^{-\frac{|x|^2}{\beta^2}}$.

The corresponding result in terms of the $L^2(\Rn)$-norm  was established by Sitaram, Sundari, and Thangavelu in \cite{SisuT}:
$$
\text{If}\;\;\;\;e^{\frac{|x|^2}{\beta^2}}\,f(x),\;\;e^{\frac{4|\xi|^2}{\alpha^2}}\widehat
f(\xi)\in L^2(\mathbb R^n),\;\;\text{and}\;\;\alpha\,\beta\leq
4,\;\text{then}\;f\equiv 0.
$$

In terms of the free solution of the Schr\"odinger equation the
$L^2$-version of Hardy Uncertainty Principle
says :
\begin{equation}
\label{hardy-l2}
\text{If}\;\;
\;e^{\frac{|x|^2}{\beta^2}}\,u_0(x),\;\;e^{\frac{|x|^2}{\alpha^2}}\,e^{it\Delta}\,u_0(x)\in
L^2(\mathbb R^n),\;\;\text{and}\;\;\alpha\,\beta\leq
4t,\;\text{then}\;u_0\equiv 0.
\end{equation}

In \cite{ekpv09} we proved the following result:

\begin{theorem}\label{Theorem A1}
Given  any solution  $u \in C([0,T] :L^2(\Rn))\,$ of
\begin{equation}
\label{ivp11}
\partial_tu=i\left(\triangle u+V(x,t)u\right),\,\;\;\;
\text{in}\;\;\, \Rn\times [0,T],
\end{equation}
with  $V=V(x,t)$ complex valued, bounded ( i.e.
$\|V\|_{L^{\infty}(\R^n\times[0,T])}\leq C$)
and
\begin{equation}
\label{14a}
\lim_{R\rightarrow +\infty}\|V\|_{L^1([0,T] : L^\infty(\Rn\setminus B_R)}=0,
\end{equation}
or $V(x,t)=V_1(x)+V_2(x,t)$ with $V_1$ real valued and $V_2$ complex valued with
$$
\sup_{t\in [0,T]}\,\|e^{k|x|^2}\,V_2(\cdot,t)\|_{L^{\infty}(\mathbb
R^n)}<\infty,\;\;\;\;\forall k\in\mathbb Z^+,
$$
satisfying that for some $t\in (0,T]$
$$
e^{\frac{|x|^2}{\beta^2}}\,u_0,\;\;\;e^{\frac{|x|^2}{\alpha^2}}\,u(x,t)\in
L^2(\mathbb R^n),
$$
with $\alpha\,\beta<4 t$,
then $\,u_0\equiv 0$.
\end{theorem}

Notice that  Theorem \ref{Theorem A1} recovers the $L^2$-version of Hardy
Uncertainty Principle \eqref{hardy-l2} for solutions of the IVP
\eqref{ivp11},
except for the limiting case $\alpha\,\beta=4t$  for which we prove that the
corresponding result fails.
More precisely, in \cite{ekpv09} it was shown that there exist (complex-valued)
bounded potentials $V(x,t)$ satisfying
\eqref{14a} for which there exist nontrivial solutions $u \in C([0,T]
:L^2(\Rn))$ of \eqref{ivp11}
satisfying
$$
e^{\frac{|x|^2}{\beta^2}}\,u_0,\;\;\;e^{\frac{|x|^2}{\alpha^2}}\,u(x,T)\in
L^2(\mathbb R^n),
$$
with $\alpha\,\beta=4T$.

\vskip.1in

This work is motivated by a different kind of uncertainty principles written
in terms of the free
solution of the Schr\"odinger equation. As it was mentioned, we are
interested in its extentions
to solutions of the equation in \eqref{E:1.1},
and  to the difference of two solutions of the nonlinear equation
\eqref{E: NLS}.

First,  one has the  result due to Beurling-H\"ormander \cite{Ho}:  If
$f\in L^1(\mathbb R)$ and
\begin{equation}
\label{beurling}
\int_{\mathbb R} \int_{\mathbb R} |f(x)| |\widehat f(\xi)| e^{
|x\,\xi|}\,dx\,d\xi<\infty, \;\;\;\text{then}
\;\;\;f\equiv 0.
\end{equation}
This  was extended to higher dimensions $n\geq 2$ in \cite{bonami2} and \cite{ray} : If $f\in L^2(\mathbb R^n), n\geq 2$ and
\begin{equation}
\label{beurlingn}
\int_{\mathbb R^n} \int_{\mathbb R^n} |f(x)| |\widehat f(\xi)| e^{
|x\,\cdot \xi|}\,dx\,d\xi<\infty, \;\;\;\text{then}
\;\;\;f\equiv 0.
\end{equation}

We observe that \eqref{beurling}, \eqref{beurlingn} implies : If 
$p\in(1,2)$, $\,1/p+1/q=1$, $\,\alpha,\,\beta>0$, and 
\begin{equation}
\label{primera}
 \int_{\mathbb R^n}|f(x)|\, e^{\frac{\alpha^p|x|^p}{p}}dx \,+
\,\int_{\mathbb R^n} |\widehat
f(\xi)| \,e^{\frac{\beta^q|\xi|^q}{q}}d\xi<\infty,\;\;\alpha \beta\geq 1\;\Rightarrow \;f\equiv 0,
 \end{equation}
 or in terms of the solution of the free Schr\"odinger
equation :
\vskip.05in
If $u_0\in L^1(\mathbb R)$ or $u_0\in L^2(\mathbb R^n)$, if $n\geq 2$,  and for
some $\,t\neq 0$
 \begin{equation}
 \label{pq}
 \int_{\mathbb R^n}\;|u_0(x)|\, e^{\frac{\alpha^p|x|^p}{p}}dx \,+
\,\int_{\mathbb R^n}\,
|\,e^{it\Delta}u_0(x)|
\,e^{\frac{\beta^q|x|^q}{q(2t)^q}}dx<\infty,\;\;\;
\alpha\,\beta\geq 1,
\end{equation}
then $\;u_0\equiv 0$. 
\vskip.1in

The following  related (and stronger in one dimension) result was
established by Bonami, Demange and Jaming \cite{bonami2} (for further results
see \cite{bonami1} and references therein): Let $f\in
L^2(\mathbb R^n)$, $1<p<2\;$ and $\,1/p+1/q=1\,$ such that for some $j=1,..,n$,
\begin{equation}
\label{bonami}
 \int_{\mathbb R^n}
|f(x)|e^{\frac{\alpha^p|x_j|^p}{p}}dx<\infty\;\;+\;\;\int_{\mathbb R^n}
|\widehat f(\xi)|e^{\frac{\beta^q|\xi_j|^q}{q}}d\xi<\infty.
\end{equation}
 If $\;\alpha\,\beta>|\cos(p\pi/2)|^{1/p}$, then
$\;f\equiv 0$. If $\,\alpha \beta <|\cos(p\pi/2)|^{1/p}$ there exist 
non-trivial functions satisfying \eqref{bonami} for all $j=1,..,n$. 
\vskip.05in

This kind of  uncertainty principles involving  conjugate exponent $p,\,q$ were first studied by G. W. Morgan in \cite{Mo}.
\vskip.1in

In \cite{GeShi} Gel'fand and Shilov considered the class $Z^p_p,\,p\geq 1$
defined as  the space of all functions $\varphi(z_1,..,z_n)$ which are
analytic for all values of $z_1,..,z_n\in \mathbb C$ and such that
$$
|\varphi(z_1,..,z_n)|\leq C_0\, e^{\sum_{j=1}^n\,\epsilon_j\,C_j\,|z_j|^p},
$$
where the $C_j,\,j=0,1,..,n$ are positive constants and $\epsilon_j=1$ for
$z_j$ non-real and $\epsilon_j=-1$ for $z_j$ real, $j=1,..,n$, and showed
that the  Fourier transform of the function space $Z_p^p$ is the space
$Z_q^q$,
with $\,1/p+1/q=1$.
\vskip.05in
Notice that  the class $Z_p^p$ with $p\geq 2$ is closed respect to
multiplication by $\,e^{i c |x|^2}$. Thus, if $u_0\in Z^p_p,\,p\geq 2$, then by \eqref{formula-l}
one has that $\,|e^{it\Delta}u_0(x)|\leq d(t)\,e^{-a(t)|x|^q}$, for some functions
$\,d,\,a\,:\,\,\mathbb R\to(0,\infty)$.

Our main result in this paper is the following:

\begin{theorem}\label{Theorem 1}
Given $\,p\in(1,2)$ there exists $\,M_p>0$ such that  for any solution  $u
\in C([0,1] :L^2(\Rn))$ of
\begin{equation*}%\label{E: initialvaluesch}
\label{E: 1.111}\partial_tu=i\left(\triangle u+V(x,t)u\right),\;\;\;\text{in}\;\;\;\;\;\Rn\times [0,1],
\end{equation*}
with  $V=V(x,t)$  complex valued,  bounded ( i.e.
$\|V\|_{L^{\infty}(\R^n\times[0,1])}\leq C$)
and
\begin{equation}
\label{14}
\lim_{R\rightarrow +\infty}\|V\|_{L^1([0,1] : L^\infty(\Rn\setminus B_R))}=0,
\end{equation}
satisfying for some constants $\,a_0,\,a_1,\,a_2>0$
\begin{equation}
\label{12}
\int_{\mathbb R^n} |u(x,0)|^2\,e^{2a_0 |x|^p}dx < \infty,
\end{equation}
and for any $k\in\Z^+$
\begin{equation}
\label{13}
\int_{\mathbb R^n} |u(x,1)|^2\,e^{2k |x|^p}dx < a_2 e^{2 a_1 k^{q/(q-p)}},
\end{equation}
$1/p+1/q=1$, if
\begin{equation}
\label{conditionp}
\,a_0\,a_1^{(p-2)} > M_p,
\end{equation}
then $\,u\equiv 0$.
\end{theorem}

\begin{corollary}\label{Corollary 2}
Given $\,p\in(1,2)$ there exists $N_p>0$ such that if
\newline $u\in C([0,1]:L^2(\mathbb R^n))$ is a solution of
$$
\partial_t u=i (\Delta u +V(x,t)u),
$$
with  $V=V(x,t)$  complex valued,  bounded ( i.e.
$\|V\|_{L^{\infty}(\R^n\times[0,1])}\leq C$) and 
$$
\lim_{R\to\infty} \,\int_0^1\,\sup_{|x|>R} |V(x,t)| dt=0,
$$
and there exist $\,\alpha,\,\beta>0$
\begin{equation}
\label{con1}
\int_{\mathbb R^n}
|u(x,0)|^2e^{2\,\alpha^p\,|x|^p/p}dx\;\,\,+\,\,\;\int_{\mathbb
R^n}
|u(x,1)|^2e^{2\,\beta^q\,|x|^q/q}dx<\infty,
\end{equation}
$\,1/p+1/q=1$, with
\begin{equation}
\label{conditionp2}
\;\alpha\,\beta > N_p,
\end{equation}
then $\;u\equiv 0$.
\end{corollary}

As a direct  consequence of  Corollary \ref{Corollary 2}  we get the
following   result regarding the uniqueness of solutions for non-linear
equations of the form \eqref{E: NLS}.

\begin{theorem}
\label{Theorem 3}

Given $\,p\in(1,2)$ there exists $\,N_p>0$ such that  if
$$
u_1,\,u_2 \in C([0,1] : H^k(\R^n)),
$$
are strong solutions of \eqref{E: NLS}  with $k\in \Z^+$, $k>n/2$,
$F:\C^2\to \C$, $F\in C^{k}$  and $F(0)=\partial_uF(0)=\partial_{\bar
u}F(0)=0$, and there exist $\,\alpha,\,\beta>0$
such that
\begin{equation}
\label{con2}
e^{\alpha^p\,|x|^p/p}\left(u_1(0)-u_2(0)\right),\;\;\;\
e^{\beta^q\,|x|^q/q}\left(u_1(1)-u_2(1)\right) \in L^2(\R^n),
\end{equation}
$1/p+1/q=1$, with
\begin{equation}
\label{conditionp2b}
\,\alpha\,\beta > N_p,
\end{equation}
then $u_1\equiv u_2$.

\end{theorem}

Notice that the conditions \eqref{conditionp} and  \eqref{conditionp2} are independent of the size of
the potential and that we do not assume any regularity on the potential $V(x,t)$.

It will be clear from our proof of Theorem \ref{Theorem 1} that  the
result in \cite{bonami2} \eqref{bonami} can be extended to our setting 
with an unsharp constant. More precisely,
\begin{corollary}\label{Corollary 4}
The results in Corollary \ref{Corollary 2} still hold with a different
constant $N_p>0$ if one replaces the hypothesis  \eqref{con1} by the one
dimensional version
\begin{equation}
\label{conn=1}
\int_{\mathbb R^n}
|u(x,0)|^2e^{2\,\alpha^p\,|x_j|^p/p}dx<\infty\,\;\,\,+\,\,\;\int_{\mathbb
R^n}
|u(x,1)|^2e^{2\,\beta^q\,|x_j|^q/q}dx<\infty
\end{equation}
for some $j=1,..,n$. 
\end{corollary}

\underline{Remarks} (i) Similarly, the non-linear version of Theorem \ref{Theorem 3} still holds, with different constant $N_p>0$,  if one replaces the hypothesis \eqref{con2} by 
$$
e^{\alpha^p\,|x_j|^p/p}\left(u_1(0)-u_2(0)\right),\;\;\;\
e^{\beta^q\,|x_j|^q/q}\left(u_1(1)-u_2(1)\right) \in L^2(\R^n),
$$
for $j=1,..,n$. 

 (ii) In this work, we do not try to give an estimate
of the universal constant $N_p$. In fact,  we may remark that the corresponding version of the sharp one dimensional condition   $\,\alpha\,\beta>|\cos(p\pi/2)|^{1/p}\,$ for \eqref{primera} established in \cite{bonami2}   is unknown 
in higher dimensions $n\geq 2$. 

(iii) We do not consider here possible versions of the limiting case $\,p=1$. One can conjecture, for example that, if $ u(x,t)$ is a solution of \eqref{E:1.1} with $u(x,0)=u_0(x)$ having compact support and $u(\cdot,t)\in L^1(e^{\epsilon |x|})$ for some $\epsilon>0$ and $t\neq 0$, then $u_0\equiv 0$.

(iv) As in some of our our previous works the main idea in the proof is to combine an
upper estimate, based on the decay hypothesis at two different times (see Lemma \ref{ultimo}), with
a lower estimate based on the
positivity of the commutator operator obtained by conjugating the equation
with the appropiate exponential weight (see Lemma \ref{Lemma ole}).
In previous works we have been able to establish the upper
bound estimates from  assumptions that at time $t=0$ and $t=1$ involve the
same weight. However, in our case (Corollary \ref{Corollary 2}) we have
different weights at time $t=0$ and $t=1$. To
overcome this difficulty, we carry out the details with the weight
$e^{a_j|x|^p},\,1<p<2$, $j=0$ at $t=0$ and $j=1$ at $t=1$, with $a_0$ fixed and
$a_1=k\in\mathbb Z^+$
as in \eqref{13}. Although the powers $\,|x|^p\,$ in the exponential are
equal at time $t=0$ and $t=1$ to apply our estimate (Lemma \ref{ultimo})
we also need  to have the same constant in front of them. To achieve this we
apply the conformal or Appel tranformation, to get  solutions and potentials, whose bounds  depend on $k\in\mathbb Z^+$.
Thus we have to  consider a  family of solutions and obtain estimates on their asymptotic value as
$k\uparrow \infty$.
\vskip.2in

Next, we shall extend the method used in the proof Theorem \ref{Theorem 1} to study 
the possible profile of the concentration blow up phenomenon in solutions of non-linear Schr\"odinger
equations 
\begin{equation}
\label{E: NLS1}
i \partial_t u + \triangle u +F(u,\overline u) u=0.
\end{equation}

To illustrate the problem consider the focussing $L^2$-critical Schr\"odinger equation
\begin{equation}
\label{E: NLS2}
i \partial_t u + \triangle u +|u|^{4/n} u=0.
\end{equation}
From the pseudo-conformal transformation one has that if $u=u(x,t)$ is a solution of \eqref{E: NLS1}, then 
\begin{equation}
\label {otrasol}
v(x,t) = \frac{e^{-i|x|^2/4(1-t)}}{(1-t)^{n/2}}\,u\left(\frac{x}{1-t},\frac{t}{1-t}\right),
\end{equation}
is also a solution of \eqref{E: NLS2}   in its domain of definition.

We recall that the pseudo-conformal transformation preserves both the space $L^2(\Rn)$ and the space $H^1(\Rn)\cap L^2(\Rn:|x|^2dx)$.
In particular, if we take $u(x,t) =e^{it}\,\varphi(x)$ the standing wave solution, i.e. $\varphi(x)$ being the 
positive ground state  of the non-linear elliptic equation 
$$
-\Delta \varphi + \varphi=|\varphi|^{4/n}\varphi,\;\;\;\;x\in\Rn,
$$
it follows that
\begin{equation}
\label{bu}
v(x,t)=\frac{e^{-i(|x|^2-4)/4(1-t)}}{(1-t)^{n/2}}\,\varphi\left(\frac{x}{1-t}\right),
\end{equation}
is a solution of \eqref{E: NLS2} which blows up at time $t=1$, i.e.
$$
\lim_{t\uparrow 1}\| \nabla\,v(\cdot,t)\|_2=\infty,
$$
and
$$
\lim_{t\uparrow  1} |v(\cdot,t)|^2=  c\, \delta(\cdot), \;\;\;\text{in the distribution sense}. 
$$

 Since it is known that the ground state $\varphi$ has exponential decay, i.e.
$$
\varphi(x)\leq b_1 e^{-b_2|x|},\;\;\;\;\;\;\;\;\;\;b_1,\,b_2>0,
$$
then one has that the blow up solution $v(x,t)$ in \eqref{bu} satisfies 
\begin{equation}
\label{bound1}
|v(x,t)|\leq \frac{1}{(1-t)^{n/2}}\,Q \left( \frac{|x|}{1-t}\right),\;\;\;\;\;\;\;\;\;t\in (-1,1),
\end{equation}
in this case with $\;Q(x)=b_1\,e^{-b_2|x|}$.

Therefore, for  \eqref{E: NLS1} with
\begin{equation}
\label{condition-power}
|F(z,\overline z)|\leq b_0 \,|z|^{\theta},\;\;\;\;\;\text{with}\;\;\;\;b_0,\,\theta>0\,\;\;\;\text{for}\,\;\;\;|z|>>1,
\end{equation}
one may ask if it is possible to have a faster  \lq\lq concentration profile"  than the one  described in \eqref{bound1}. More precisely, whether or not \eqref{bound1} can hold  with
\begin {equation}
\label {bound2}
Q(x)=b_1\,e^{-b_2|x|^p},\;\;\;\;\;\;\;b_1,\,b_2>0,\;\;\;p>1.
\end{equation}

Our next result shows that this is not the case at least for $p>p(\theta)$.

\begin{theorem}
\label{Theorem 4}

Let $u\in C((-1,1):L^2(\Rn))$ be a solution of the equation \eqref{E: NLS1}
with $|F(z,\overline z)|$ as in \eqref{condition-power}. 
Assume that the $L^2$-norm of the solution $u(x,t)$ is preserved
\begin{equation}
\label{preserved-2}
\|u(\cdot,t)\|_2=\|u(\cdot,0)\|_2=\|u_0\|_2=a,\;\;\;\,t\in(-1,1),
\end{equation}
and that \eqref{bound1} holds with $\;Q(\cdot)$ as in \eqref{bound2}. 
If $p>p(\theta)=2(\theta n-2)/(\theta n-1)$, then $a=0$.

\end{theorem}

\underline{Remarks} (i) We shall restrict to the case $p\in(1,2)$, and observe that if $\theta=4/n$ 
then $\,p(\theta)=4/3$.
The value $\,4/3\,$ is related with  the following result due to  V. Z. Meshkov 
   \cite{Mes}: Let  $\,w\in H^2_{loc}(\Rn)$ be a solution of 
$$
\Delta w- V(x)w = 0,\;\;\;\;x\in\Rn,\;\;\;\text{with} \;\;\;V\in L^{\infty}(\R^n). 
$$

$$
\text{If}\;\;\,\int |w(x)|^2 \,e^{2a|x|^{4/3}} dx <\infty, \;\;\;\forall \,a>0,\;\;\text{then}\;\;\;\;w\equiv 0.
$$

It was also proved in \cite{Mes} that for complex valued potentials $V$ the exponent $4/3$  is
sharp. For further comments see the remark after the proof of Theorem \ref{Theorem 4}.

\vskip.1in

   The rest of this paper is organized as follows. In section 2 we establish the upper bounds needed
in the proof of Theorem \ref{Theorem 1}. The lower bounds as well as the conclusion of the proof 
of Theorem \ref{Theorem 1} are 
carried out  in  section 3. Corollary \ref{Corollary 2}  and Theorem \ref{Theorem 3} are proved  in section 4. Section 5 contains the proof of Theorem \ref{Theorem 4}. Finally, in the Appendix we establish some
identities used in the paper.

\end{section}
%%%%%%%%%%%%%%%%%%%%%%%%%%%%%%
\begin{section}{Proof of Theorem \ref{Theorem 1} : Upper bounds}\label{S1}
%%%%%%%%%%%%%%%%%%%%%%

In this paper $c_n $ will denote  a constant which may depend only on the
dimension $n$, which  may change from line to line. Similarly, $c_p$ will
denote a constant depending only on the values of $p$ and $n$,  and in sections 2-3 $\,c^*$
will denote a constant depending of the initial parameters, i.e.  the norms of $u$ and
$V$ in the hypothesis, and on the values of $ p,\,n,\,a_0$ and $\,a_1$, 
whose exact value will be irrelevant to our estimate when we take $\,k\,$
tending to infinity.

We recall the conformal or Appell transformation.
If   $u(y,s)$ verifies
\begin{equation}%\label{E: parab&#151;licageneral}
\label{2.1}
\partial_su=i\left(\triangle
u+V(y,s)u+F(y,s)\right),\;\;\;\;\;\;\;(y,s)\in \R^n\times [0,1],
\end{equation}
and $\alpha$ and $\beta$ are positive,  then
\begin{equation}
\label{2.2}
\widetilde u(x,t)=\left(\tfrac{\sqrt{\alpha\beta}}{\alpha(1-t)+\beta
t}\right)^{\frac n2}u\left(\tfrac{\sqrt{\alpha\beta}\,
x}{\alpha(1-t)+\beta t}, \tfrac{\beta t}{\alpha(1-t)+\beta
t}\right)e^{\frac{\left(\alpha-\beta\right) |x|^2}{4i(\alpha(1-t)+\beta
t)}},
\end{equation}
verifies
\begin{equation}%\label{E: parab&#151;licageneral2}
\label{2.3}
\partial_t\widetilde u=i\left(\triangle \widetilde u+\widetilde
V(x,t)\widetilde u+\widetilde F(x,t)\right),\;\; \text{in}\  \R^n\times
[0,1],
\end{equation}
with
\begin{equation}
\label{potencial}
\widetilde V(x,t)=\tfrac{\alpha\beta}{\left(\alpha(1-t)+\beta
t\right)^2}\,V\left(\tfrac{\sqrt{\alpha\beta}\, x}{\alpha(1-t)+\beta t},
\tfrac{\beta t}{\alpha(1-t)+\beta t}\right),
\end{equation}
and
\begin{equation}
\label{externalforce}
\widetilde F(x,t)=\left(\tfrac{\sqrt{\alpha\beta}}{\alpha(1-t)+\beta
t}\right)^{\frac n2+2}F\left(\tfrac{\sqrt{\alpha\beta}\,
x}{\alpha(1-t)+\beta t}, \tfrac{\beta t}{\alpha(1-t)+\beta
t}\right)e^{\frac{\left(\alpha-\beta\right) |x|^2}{4i(\alpha(1-t)+\beta
t)}}.
\end{equation}

In our case, we shall chose $\beta=\beta(k)$. By hypothesis
\begin{equation}
\label{hyp}
\begin{aligned}
&\|e^{a_0|x|^p}u(x,0)\|_2\equiv A_0,\\
&\|e^{k|x|^p}u(x,1)\|_2\equiv A_k\leq a_2\,e^{a_1k^{q/(q-p)}}=
a_2\,e^{a_1k^{1/(2-p)}}.
\end{aligned}
\end{equation}

Thus, for $\gamma=\gamma(k)\in [0,\infty)$ to be chosen later, one has
\begin{equation}
\label{20}
\|e^{\gamma|x|^p}\widetilde
u_k(x,0)\|_2=\|e^{\gamma\left(\frac{\alpha}{\beta}\right)^{p/2}|x|^p}u(x,0)\|_2 =B_0,
\end{equation}
and
\begin{equation}
\label{21}
\|e^{\gamma|x|^p}\widetilde
u_k(x,1)\|_2=\|e^{\gamma\left(\frac{\beta}{\alpha}\right)^{p/2}|x|^p}u(x,1)\|_2=A_k.
\end{equation}

To match our hypothesis we take
$$
\gamma\,\left(\frac{\alpha}{\beta}\right)^{p/2}=a_0\;\;\;\;\;\text{and}\;\;\;\;\;\gamma\,\left(\frac{\beta}{\alpha}\right)^{p/2}=k.
$$

Therefore,
\begin{equation}
\label{constants}
\gamma=(k\,a_0)^{1/2},\;\;\;\;\;\;\;\;\;\;\beta=k^{1/p},\;\;\;\;\;\;\;\;\;\;
\alpha=a_0^{1/p}.
\end{equation}

From \eqref{2.3}, defining 
\begin{equation}
\label{definitionofM}
M=\int_0^1\|\Im V(t)\|_{\infty}dt=\int_0^1\|\Im \widetilde
V(s)\|_{\infty}ds ,
\end{equation}
it follows, using energy estimates, that
\begin{equation}
\label{energy}
\|u(0)\|_2\,e^{-M}\leq \|u(t)\|_2=\|\,\widetilde u(s)\|_2 \leq
\|u(0)\|_2\,e^M,\;\;\;t,\,s\in[0,1].
\end{equation}
where $\,s=\beta t/(\alpha(1-t)+\beta t)$.

\vskip.1in

Next, we shall  combine the estimate : for any $x\in \R^n$
\begin{equation}
\label{est1}
e^{\gamma |x|^p/p}\simeq \int_{\R^n} \,e^{\gamma^{1/p}\lambda\cdot
x-|\lambda|^q/q}\,
|\lambda|^{n(q-2)/2}\,d\lambda,
\end{equation}
whose proof will be given in the appendix, with the following result found
in  \cite{kpv02} :
\vskip.1in

\begin{lemma}\label{ultimo} There exists $\epsilon_0>0$ such that if
\begin{equation}
\label{hyp2}
\mathbb V:\mathbb R^n\times [0,1]\to\mathbb C,\;\;\;\;\text{with}\;\;\;\;
\|\mathbb V\|_{L^1_tL^{\infty}_x}\leq \epsilon_0,
\end{equation}
and $u\in C([0,1]:L^2(\mathbb R^n))$ is a strong solution of the IVP
\begin{equation}
\begin{cases}
\begin{aligned}
\label{eq1}
&\partial_tu=i(\Delta +\mathbb V(x,t))u+\mathbb F(x,t),\\
&u(x,0)=u_0(x),
\end{aligned}
\end{cases}
\end{equation}
with
\begin{equation}
\label{hyp3} u_0,\,u_1\equiv u(\,\cdot\,,1)\in
L^2(e^{2\lambda\cdot x}dx),\;\mathbb F\in L^1([0,1]:L^2(e^{2\lambda\cdot
x}dx)),
\end{equation}
for some $\lambda\in\mathbb R^n$, then there exists $c_n$ independent of
$\lambda$ such that
\begin{equation}
\label{uno}
\sup_{0\leq t\leq 1}\| e^{\lambda\cdot x} u(\,\cdot\,,t)\|_2 \leq c_n
\Big(\|e^{\lambda\cdot x} u_0\|_2 + \|e^{\lambda\cdot x} u_1\|_2 +\int_0^1
\|e^{\lambda\cdot x}\, \mathbb F(\cdot, t)\|_2 dt\Big).\;\;\square
\end{equation}
\end{lemma}

\vskip.1in

We want to apply Lemma \ref{ultimo} to a solution of the equation
\eqref{2.3}. Since $0<\alpha<\beta=\beta(k)$ for $k\geq k_0(c^*)$ it
follows
that for any $t\in[0,1]$
$$
\alpha \leq \alpha(1-t) +\beta t\leq \beta,
$$
therefore if $\,y=\sqrt{\alpha \beta}\, x/({\alpha(1-t) +\beta t})$, then
from \eqref{constants}
\begin{equation}
\label{xy}
 \sqrt{\frac{\alpha}{\beta}} \,|x|=\frac{a_0^{1/2p}}{k^{1/2p}}\, |x|\leq |y|\leq \sqrt{\frac{\beta}{\alpha}} \,|x|=\frac{k^{1/2p}}{a_0^{1/2p}}\, |x|.
\end{equation}
Thus,
\begin{equation}
\label{potencial1}
\left|\tfrac{\alpha\beta}{\left(\alpha(1-t)+\beta
t\right)^2}\,V\left(\tfrac{\sqrt{\alpha\beta}\, x}{\alpha(1-t)+\beta t},
\tfrac{\beta t}{\alpha(1-t)+\beta t}\right)\right|
\leq
\frac{\beta}{\alpha}\|V\|_{\infty}=\left(\frac{k}{a_0}\right)^{1/p}\|V\|_{\infty},
\end{equation}
and so
\begin{equation}
\label{A1}
\|\widetilde V\|_{\infty}\leq \left(\frac{k}{a_0}\right)^{1/p}\|V\|_{\infty}.
\end{equation}
Also, if
\begin{equation}
\label{st}
\begin{aligned}
& s=\beta t/(\alpha(1-t) +\beta t)\\
\\
&\frac{ds}{dt}=\frac{\alpha \beta}
{(\alpha(1-t) +\beta
t)^2},\;\;\;\;\;\text{or}\;\;\;\;dt=\frac{(\alpha(1-t) +\beta t)^2}{\alpha
\beta} ds.
\end{aligned}
\end{equation}

Therefore,
\begin{equation}
\label{eneene}
\int_0^1\|\,\widetilde V(\cdot,
t)\|_{L^{\infty}}dt=\int_0^1\|V(\cdot,s)\|_{L^{\infty}}ds,
\end{equation}
and from \eqref{xy}
$$
\int_0^1\|\,\widetilde V(\cdot, t)\|_{L^{\infty}(|x|\geq
R)}dt \leq \int_0^1\|V(\cdot,s)\|_{L^{\infty}(|y|>\Upsilon)}ds,
$$
where by hypothesis
$$
\Upsilon=\left(\frac{a_0}{k}\right)^{1/2p} R.
$$
Thus, if
$$
\int_0^1\|V(\cdot,s)\|_{L^{\infty}(|y|>\Upsilon)}ds \leq \epsilon_0,
$$
then
\begin{equation}
\label{oo}
\int_0^1\|\,\widetilde V(\cdot, t)\|_{L^{\infty}(|x|\geq R)}dt\leq
\epsilon_0,\;\;\;\;\;\;\;\,R= \Upsilon\,\left(\frac{k}{a_0}\right)^{1/2p},
\end{equation}
and
we can apply Lemma \ref{ultimo} to the equation \eqref{2.3} with
\begin{equation}
\label{split}
\mathbb V=\widetilde V\,\chi_{(|x|>R)}(x),\;\;\;\;\;\;\;\;\;\;\mathbb
F=\widetilde V\,\chi_{(|x|\leq R)}(x)\,\widetilde u(x,t),
\end{equation}
 to get
\begin{equation}
\label{pre-pre}
\begin{aligned}
&\sup_{[0,1]}\| e^{(2p)^{1/p}\gamma^{1/p}\lambda\cdot x/2}\, \widetilde u
(t)\|_2\\
\\
&\leq c_n \left(\| e^{(2p)^{1/p}\gamma^{1/p}\lambda\cdot x/2}\, \widetilde
u (0)\|_2
+\| e^{(2p)^{1/p}\gamma^{1/p}\lambda\cdot x/2} \,\widetilde u
(1)\|_2\right )\\
\\
& + c_n \|\widetilde V\|_{\infty}
\|u(0)\|_2\,e^M\,e^{|\lambda| (2p)^{1/p}\gamma^{1/p} R/2}.
\end{aligned}
\end{equation}

Now, we square \eqref{pre-pre} to find that
$$
\aligned
& \sup_{[0,1]}\int e^{(2p)^{1/p}\gamma^{1/p}\lambda\cdot x}\,|\widetilde
u(x,t)|^2dx\\
&\leq c_n \int e^{(2p)^{1/p}\gamma^{1/p}\lambda\cdot x}\,(|\widetilde
u(x,0)|^2+|\widetilde u(x,1)|^2)dx\\
&+
c _n\|\widetilde V\|^2_{\infty} \,
\|u(0)\|^2_2\,e^{2M}\,e^{|\lambda| (2p)^{1/p}\gamma^{1/p} R},
\endaligned
$$
and multiply the above inequality (for a fixed $t$) by
$\,e^{-|\lambda|^q/q}\,|\lambda|^{n(q-2)/2}$, integrate in $\lambda$ and
in $x$, use Fubini  theorem and the formula \eqref{est1} to
obtain
\begin{equation}
\label{est2a}
\begin{aligned}
&\int_{|x|>1}e^{2\gamma|x|^p}|\widetilde u(x,t)|^2dx
\leq c_n \int e^{2\gamma|x|^p}(|\widetilde u(x,0)|^2+|\widetilde
u(x,1)|^2)dx\\
\\
& + c_n \|u(0)\|^2_2 \,\|\widetilde V\|_{\infty}^2 \,R^{c_p}\,e^{2M}
e^{2\gamma\,R^p}.
\end{aligned}
\end{equation}

\vskip.08in

Hence,  \eqref{hyp}, \eqref{constants}, \eqref{energy}, \eqref{A1}, and
\eqref{est2a} lead to

\begin{equation}
\label{est2}
\begin{aligned}
&\sup_{[0,1]} \|e^{\gamma |x|^p} \widetilde u(t)\|_2 \leq c_n(\|e^{\gamma
|x|^p} \widetilde u(0)\|_2
+\|e^{\gamma |x|^p} \widetilde u(1)\|_2)\\
&+c_n \|u(0)\|_2 \,e^{M} e^{\gamma}+c_n \|u(0)\|_2\,
\left(\frac{k}{a_0}\right)^{c_p}\|V\|_{\infty} \,e^{M}
e^{\Upsilon^p\,\gamma(k/a_0)^{1/2}}\\
\\
&\leq c_n(A_0+A_k)+c_n \|u(0)\|_2
\,e^{M}\,\left(\,e^{\gamma}+\left(\frac{k}{a_0}\right)^{c_p}\|V\|_{\infty}
\,e^{\Upsilon^p\, k}\right),\\
\\
&\leq c^*\,A_k= c^*\,e^{a_1\,k^{1/(2-p)}},
\end{aligned}
\end{equation}
for $k\geq k_0(c^*)$ sufficiently large, since $\,1/(2-p)>1$.

\vskip.15in

Next, we shall obtain bounds for the $\,\nabla\widetilde u$.
Let
\begin{equation}
\label{771}
\widetilde \gamma=\gamma/2,
\end{equation}
and $\varphi$ be a strictly convex function on compact sets of $\R^n$,
radial such that (see \cite{ekpv08b} Lemma 3)
\begin{equation}
\label{772}
\begin{aligned}
& D^2\varphi \geq p(p-1) |x|^{(p-2)}I,\;\;\;\;\text{for}\;\;\;|x|\geq 1,\\
& 0\leq \varphi,\;\;\;\;\|\partial^{\alpha}\varphi\|_{L^{\infty}}\leq
c\;\;\;2\leq |\alpha|\leq 4,\;\;\|\partial^{\alpha}\varphi\|_{L^{\infty}(|x|\leq 2)}\leq c\;\;
\;\;|\alpha|\leq 4,\\
&\varphi(x)=|x|^p+O(|x|),\;\;\;\text{for}\;\;\;|x|>1.
\end{aligned}
\end{equation}

We shall use the equation
\begin{equation}
\label{new}
\partial_t\widetilde u= i\Delta \widetilde u +i
F,\;\;\;\;\;\;\;\;F=\widetilde V\, \widetilde u,
\end{equation}
and let
$$
f(x,t)=e^{\widetilde \gamma \varphi}\,\widetilde u(x,t).
$$
Then $f$ verifies (see Lemma 3 in \cite{ekpv08b})
\begin{equation}\label{E: loquefcumple}
\partial_tf=\mathcal Sf+\mathcal Af +i\,e^{\widetilde \gamma\varphi}F,\;\;\;\;
\text{in}\;\;\;\R^n\times[0,1],
\end{equation}
with symmetric and skew-symmetric operators $\mathcal S$ and $\mathcal A$
\begin{equation}
\label{E: formulaoperadores}
\begin{aligned}
\mathcal
S=&-i\widetilde \gamma\left(2\nabla\varphi\cdot\nabla+\triangle\varphi\right),\\
\mathcal A= &i\left(\triangle +\widetilde \gamma^2|\nabla\varphi|^2\right).
\end{aligned}
\end{equation}
A calculation shows that (see (2.14) in \cite{ekpv08b}),
\begin{equation}\label{E: formulaconmutadorindependientetiempo}
\mathcal S_t+\left[\mathcal S,\mathcal A\right]=-\widetilde \gamma
\left[4\nabla\cdot\left(D^2\varphi\nabla\,\,\,\right)-4\widetilde
\gamma^2D^2\varphi\nabla\varphi\cdot\nabla\varphi+\triangle^2\varphi\right].
\end{equation}

By Lemma 2 in \cite{ekpv08b}
\begin{equation}
\label{E: derivadasegunda}
\begin{aligned}
\partial_t^2H=\partial_t^2\left( f, f\right)=& 2\partial_t\text{\it
Re}\left(\partial_tf-\mathcal Sf-\mathcal Af,f\right)+2\left(\mathcal
S_tf+\left[\mathcal S,\mathcal A\right]f,f\right)\\
&+\|\partial_tf-\mathcal Af+\mathcal Sf\|^2-\|\partial_tf-\mathcal
Af-\mathcal Sf\|^2,
\end{aligned}
\end{equation}
so
\begin{equation}
\label{E-b}
\begin{aligned}
\partial_t^2H& \geq  2\partial_t \text{\it Re}\left(\partial_tf-\mathcal
Sf-\mathcal Af,f\right)\\
&+2\left(\mathcal S_tf+\left[\mathcal S,\mathcal
A\right]f,f\right)-\|\partial_tf-\mathcal Af-\mathcal Sf\|^2.
\end{aligned}
\end{equation}

Multiplying \eqref{E-b} by $t(1-t)$ and integrating in $t$ we obtain
\begin{equation}
\label{EH}
\begin{aligned}
&2\int_0^1t(1-t)\left(\mathcal S_tf+\left[\mathcal S,\mathcal
A\right]f,f\right)dt\\
&\leq c_n\,\sup_{[0,1]}\|e^{\widetilde\gamma\,\varphi}
\,\widetilde u(t)\|_2^2+c_n\,\sup_{[0,1]}\|e^{\widetilde\gamma\,\varphi}
F(t)\|^2_2.
\end{aligned}
\end{equation}
This computation can be justified by parabolic regularization using the
fact that we already know the decay estimate for $\,\widetilde u$, (see the proof of Theorem 5 in  \cite{ekpv08b}). Hence, combining \eqref{est2}, \eqref{constants} and \eqref{A1} it follows that

\begin{equation}
\label{aa1}
\begin{aligned}
&8 \,\widetilde \gamma \,\int_0^1\int \,t(1-t)D^2\varphi \nabla f\cdot
\nabla f dx dt \\
\\
&+
8\, \widetilde \gamma^3 \,\int_0^1\int \,t(1-t)D^2\varphi \nabla
\varphi\cdot\nabla\varphi \,|f|^2dxdt\\
\\
&\leq c_n\,\sup_{[0,1]}\|e^{\widetilde\gamma\,\varphi}\,\widetilde u(t)\|_2^2
+c_n\,\|\widetilde
V\|_{L^{\infty}}^2\,\sup_{[0,1]}\|e^{\widetilde\gamma\,\varphi}
\,\widetilde u(t)\|_2^2 +c_n \widetilde \gamma
\sup_{[0,1]}\|e^{\widetilde\gamma\,\varphi}\,\widetilde u(t)\|_2^2\\
\\
&\leq c^*\,k^{c_p}\,A_k^2.
\end{aligned}
\end{equation}
\vskip.1in

We recall that
$$
\nabla f=\widetilde \gamma \,\nabla\varphi \,e^{\widetilde\gamma
\varphi}\,\widetilde u+
e^{\widetilde \gamma\varphi}\,\nabla\widetilde u,
$$
and notice that
$$
|\,\widetilde \gamma^3 \,D^2\varphi\,\nabla \varphi\cdot\nabla\varphi
\,e^{2\widetilde \gamma \,\varphi}|\leq c_p \widetilde \gamma^3
\,e^{2\gamma |x|^p}.
$$

Hence, using that
$$
D^2\varphi\,\nabla \varphi\cdot\nabla\varphi \,e^{2\widetilde \gamma
\,\varphi}
\leq c_n (1+|x|)^{p-2+2(p-1)} \,e^{2\widetilde \gamma \,\varphi}
\leq c_p e^{3 \gamma \,\varphi/2}
$$
we can conclude that
\begin{equation}
\label{estabove}
\begin{aligned}
&\gamma\,\int_0^1\int t(1-t)\frac{1}{(1+|x|)^{2-p}}\,|\nabla \widetilde
u(x,t)|^2 e^{\gamma |x|^p}dxdt
+\sup_{[0,1]}\,\|e^{\gamma|x|^p/2}\,\widetilde u(t)\|_2^2\\
\\
&\leq c^*\,k^{c_p}\,A_k^2= c^*\,k^{c_p}\,e^{2 a_1\,k^{q/(q-p)} } =
c^*\,k^{c_p}\,e^{2 a_1\,k^{1/(2-p)}},
\end{aligned}
\end{equation}
for $k\geq k_0(c^*)$ sufficiently large.

\end{section}

%%%%%%%%%%%%%%%%%%%%%%%%%%%%%%
\begin{section}{Proof of Theorem \ref{Theorem 1} : Lower bounds and conclusion. }\label{S:  ppp}
%%%%%%%%%%%%%%%%%%%%%%

First, we deduce a lower bound for
\begin{equation}
\label{lower1}
\Phi \equiv  \int_{|x|<R/2}\,\int_{3/8}^{5/8} |\widetilde u(x,t)|^2dt
dx\geq  \|u(0)\|_2\,e^{-M}/10,
\end{equation}
for $R$ sufficiently large. From (2.2)
$$
\aligned
&\Phi=\int_{|x|<R/2}\,\int_{3/8}^{5/8} |\widetilde u(x,t)|^2dt dx\\
\\
&=\int_{|x|<R/2}\,\int_{3/8}^{5/8}
\left| \left(\tfrac{\sqrt{\alpha\beta}}{\alpha(1-t)+\beta t}\right)^{\frac
n2}u\left(\tfrac{\sqrt{\alpha\beta}\, x}{\alpha(1-t)+\beta t},
\tfrac{\beta t}{\alpha(1-t)+\beta t}\right)\right|^2dtdx\\
\\
&\geq c_n\,\frac{\beta}{\alpha}\,\int_{|y|\leq R\,(a_0/k)^{1/2p}}\;
\int_{s(3/8)}^{s(5/8)} \,|u(y,s)|^2\,\frac{ds dy}{s^2}\\
\\
&\geq
c_n\,\frac{\beta}{\alpha}\,\int_{|y|\leq R\,(a_0/k)^{1/2p}}\;
\int_{s(3/8)}^{s(5/8)} \,|u(y,s)|^2ds dy,
\endaligned
$$

\noindent where in the $t$ variable we have used that in the  interval
$t\in [3/8,5/8]$  (see \eqref{st})
$$
dt = \frac{\beta}{\alpha}\,\frac{t^2}{s^2} \,ds\sim
\frac{\beta}{\alpha}\,\frac{1}{s^2} \,ds,
$$
since $s(t)=\beta t/(\alpha(1-t)+\beta t)$,
$$
s(5/8)-s(3/8)=\frac{\alpha\,\beta\,(5/8-3/8)}{(\alpha \,3/8 +\beta
\,5/8)(\alpha \,5/8+\beta\, 3/8)}
\sim \frac{\alpha}{\beta},
$$
for $k \geq c_n$, and $s(5/8)> s(3/8) \uparrow 1$ as $k\uparrow \infty$  with
$s(3/8)\geq 1/2$ for $k\geq c_n$, and in the $x$ variable that
$$
y=\sqrt{\alpha\,\beta} x/(\alpha(1-t)+\beta t),
$$
so for $\,t\in [3/8,5/8]$
$$
y\sim \sqrt{\frac{\alpha}{\beta}}
\,|x|=\left(\frac{a_0}{k}\right)^{1/2p}\,|x|.
$$

Thus, taking
\begin{equation}
\label{ourR}
R \geq  \iota (k/a_0)^{1/2p},
\end{equation}
with $\,\iota=\iota(u) $ a constant to be determined, it follows that
$$
\Phi\geq c_n\,\frac{\beta}{\alpha}\,\int_{|y|\leq \iota}\;\int_{I}
\,|u(y,s)|^2ds dy,
$$
where the interval  $I=I_k=[s(3/8),s(5/8)]$ satisfies $\, I\subset [1/2,1]$ and $\,|I|\sim
\alpha/\beta$ for $k$ sufficiently large. Moreover, given $\epsilon>0$ there exists $k_0(\epsilon)>0$ such that 
 for any $k\geq k_0$ one has that $\,I_k\subset [1-\epsilon,1]$.

By hypothesis on $u(x,t)$, i.e. the continuity of $\,\|u(\cdot, s)\|_2$ at $s=1$, it follows that there exists $\,\iota>>1$ and $K_0=K_0(u)$ such that for any $k\geq K_0$ 
and for any $\,s\in I_k$
$$
\int_{|y|\leq \iota}\,|u(y,s)|^2dy\geq \|u(0)\|_2\,e^{-M}/10,
$$
which yields the desired result. Below we will fix $R\sim
k^{1/2(2-p)}>>k^{1/2p},\;p>1$ (see \eqref{ourrealR}), so we could have taken $\,\iota\sim
k^l,\;l=1/2(2-p)-1/2p=1/p(2-p)>0$, and take $\,\iota$ independent of $u$
when $k\uparrow \infty$.

\vskip.2in
Next, we deduce an upper bound for
\begin{equation}
\label{upper1}
\int_{|x|<R}\,\int_{1/32}^{31/32} (|\widetilde u|^2 + |\nabla \widetilde
u|^2)(x,t)dt dx.
\end{equation}

From \eqref{energy}
and the fact that
$$
\|\widetilde u(t)\|_2=\|u(s)\|_2,\;\;\;\;\;
s=\beta t/(\alpha(1-t)+\beta t),
$$
we have
$$
\int_{|x|<R}\,\int_{1/32}^{31/32} |\widetilde u(x,t)|^2 dt dx\leq
\|u(0)\|_2^2\,e^{2M},
$$
and from \eqref{estabove}
\begin{equation}
\label{above+}
\begin{aligned}
&\int_{1/32}^{31/32}\int_{|x|\leq R}|\nabla \widetilde u(x,t)|^2dxdt\\
\\&\leq
c_n \int_{1/32}^{31/32}\int_{|x|\leq R}
t(1-t)\frac{(1+|x|)^{2-p}}{(1+|x|)^{2-p}}\,
e^{\gamma |x|^p}\,|\nabla \widetilde u(x,t)|^2 dxdt\\
\\
&\leq c_n \gamma^{-1}\,R^{2-p} \,c^*\,k^{c_p}\,A^2_k\leq
c^*\,k^{c_p}\,e^{2 a_1\,k^{q/(q-p)} } = c^*\,
k^{c_p}\,e^{2 a_1\,k^{1/(2-p)}},
\end{aligned}
\end{equation}
using that $R$ is a power of $k$. Hence,
\begin{equation}
\label{conclu}
\int_{|x|<R}\,\int_{1/32}^{31/32} (|\widetilde u|^2 + |\nabla \widetilde
u|^2)(x,t)dt dx\leq c^*\,
k^{c_p} \,e^{2 a_1\,k^{1/(2-p)}},
\end{equation}
for $k\geq k_0(c^*)$ sufficiently large.

We now recall Lemma 3.1 in \cite{ekpv06}. 

\begin{lemma}
 \label{Lemma ole}
Assume that $R>0$ and $\varphi:[0,1]\longrightarrow\R$ is a smooth function. Then, there exists $c=c(n, \|\varphi'\|_{\infty}+\|\varphi''\|_{\infty})>0$ such that, the inequality
\begin{equation}
\label{previous}
\frac{\sigma^{3/2}}{R^2}\, \|e^{\sigma |\frac xR+\varphi(t)e_1|^2}g\|_{L^2(dxdt)}\leq
c\| e^{\sigma |\frac xR+\varphi(t)e_1|^2}(i \partial_t+\Delta)g\|_{L^2(dxdt)}
\end{equation}
holds, when $\sigma \ge cR^2$ and $g\in C_0^\infty(\R^{n+1})$ has its support contained in the set \[\{(x,t): |\tfrac xR+\varphi(t)e_1|\ge 1\}\ .\]
\end{lemma}
\vskip.1in

First, we need to
show that
\begin {equation}
\label{claim32}
R>>\|\widetilde V\|_{L^{\infty}(\R^n\times [1/32, 31/32]}.
\end{equation}
But in this domain from \eqref{potencial} one sees that
$$
|\widetilde V(x,t)|\leq (32)^2\,\frac{\alpha}{\beta}\,\|V\|_{\infty} \leq
(32)^2\,\frac{a_0^{1/p}}{k^{1/p}}\,\|V\|_{\infty},
$$
from  \eqref{ourR} it is clear that  \eqref{claim32} holds. Define
\begin{equation}
\label{delta}
\delta(R)= \left(\,\int_{1/32}^{31/32}\int_{R-1\leq|x|\leq
R}\,(|\widetilde u|^2+\nabla \widetilde u|^2)(x,t)dxdt\right)^{1/2},
\end{equation}
We choose  $\varphi\in C^\infty([0,1])$ and  $\theta_R,\;\theta \in
C^{\infty}_0(\R^n)$ functions verifying
$$
0\le \varphi \le 3,\;\;\;\varphi(t) =3, \;\;\text{if}\;\,t\in [3/8,
5/8],\;\;\;\varphi(t)=0,\;\;\text{if}\;\;t\in [0,1/4]\cup
[3/4,1],
$$
$$\
\theta_R(x)=1,\;\;\text{if}\;\;|x|\leq R-1,
\;\;\;\;\theta_R(x)=0,\;\;\text{if}\;\;|x|>R,
$$
and
$$
\theta (x)=0,\;\;\text{if}\;\;|x|\leq 1,\;\;\;\;\theta
(x)=1,\;\;\text{if}\;\;|x|\ge 2.
$$
We define
\begin{equation}
\label{gg}
g(x,t)=\theta_R(x)\,\theta(\tfrac xR+\varphi(t)e_1)\,\widetilde u(x,t),
\end{equation}
and make the following remarks on $g(x,t)$:
\vskip.1in
- if $|x|\leq R/2,\;t\in[3/8,5/8]$, then $|x/R+\varphi(t)e_1|\geq
3-1/2=\frac{5}{2}>2$,  \newline

so $\;g(x,t)=\widetilde u(x,t)$,
and in this set $\;e^{\sigma\,|x/R+\varphi(t)e_1|^2}\geq e^{25 \sigma/4}$.
\vskip.1in
-if $|x|\geq R\;$ or $\;t\in [0,1/4]\cup [3/4,1]$, then $\,g(x,t)=0$, so
\newline
\vskip.01in
$\;supp\,g\subseteq \{|x|\leq R\times[1/4,3/4]\}\cap\{|x/R+\varphi(t)e_1|\geq 1\}$.
\vskip.1in
Then, if $\xi=x/R+\varphi(t)e_1$ we have
\begin{equation}
\label{3.13}
\begin{aligned}
&(i\partial_t +\Delta+\widetilde
V)g=[\theta(\xi)(2\nabla\theta_R(x)\cdot\nabla \widetilde u+
\widetilde  u\Delta
\theta_R(x) )+  2\nabla \theta(\xi)\cdot\nabla \theta_R\, \widetilde u\,]\\
&+\theta_R(x)\left[2R^{-1}\nabla\theta(\xi)\cdot\nabla \widetilde
u+R^{-2}\widetilde u\Delta\theta(\xi) +i\varphi'\partial_{x_1}\theta(\xi) \widetilde
u \right]\,=B_1+B_2.
\end{aligned}
\end{equation}
Note that
$$
supp\;B_1\subset\{(x,t)\in\R^n\times[0,1]\;:R-1\leq|x|\leq R\times[1/32,
31/32]\},
$$
and
$$
supp\;B_2\subset\{(x,t)\in\R^n\times[0,1]\;:\;1\leq|x/R+\varphi(t)e_1|\leq2\}.
$$
\vskip.1in

Now applying Lemma \ref{Lemma ole} choosing 
$$
\sigma=d_nR^2,\;\;\;\;\;\;d_n^2\geq \|\varphi''\|_{\infty}+\|\varphi'\|^2_{\infty}, 
$$
it follows that
\begin{equation}
\label{3.1}
\begin{aligned}
&R\, \|e^{\sigma |\frac xR+\varphi(t)e_1|^2}g\|_{L^2(dxdt)}\leq
c_n\| e^{\sigma |\frac xR+\varphi(t)e_1|^2}(i
\partial_t+\Delta)g\|_{L^2(dxdt)}
\\
\\&\leq
c_n \| e^{\sigma |\frac xR+\varphi(t)e_1|^2}\,\widetilde V g\|_{L^2(dxdt)}
+ c_n\| e^{\sigma |\frac xR+\varphi(t)e_1|^2}\,B_1\|_{L^2(dxdt)}\\
\\ &+
c_n \| e^{\sigma |\frac xR+\varphi(t)e_1|^2}\,B_2\|_{L^2(dxdt)}\equiv D_1+D_2+D_3.
\end{aligned}
\end{equation}

Since $R>>\|\widetilde V\|_{\infty}$, we can absorb $D_1$ in the left
hand side of \eqref{3.1}. 
On the support of $B_1$ one has $|x/R+\varphi(t)e_1|\leq 4$, thus
$$
D_2\leq c_n\,\delta(R)\,e^{16 \sigma}.
$$

On the support of $B_2$,  one has $|x|\leq R,\;t\in[1/32,31/32]$, and
\newline
$1\leq|x/R+\varphi(t)e_1|\leq2$, so
$$
D_3\leq c_n \,e^{4\,\sigma}\|\,|\widetilde u|+|\nabla \widetilde
u|\|_{L^2(|x|\leq R\times[1/32, 31/32])}.
$$

Combining this information, \eqref{lower1}, and \eqref{conclu} we have
\begin{equation}
\label{casi}
\begin{aligned}
& c_n\,e^{25 \sigma/4}\, \|u(0)\|_2\,e^{-M}\leq
R\,e^{25 \sigma/4}\,(\int_{|x|<R/2}\,\int_{3/8}^{5/8} |\widetilde
u(x,t)|^2dt dx)^{1/2}\\
\\&\leq c_n\,\delta(R)\,e^{16 \sigma} +c_n\,e^{4\,\sigma}\| |\widetilde
u|+|\nabla \widetilde u|\|_{L^2(|x|\leq R\times[1/32, 31/32])}\\
\\
&\leq c_n\,\delta(R)\,e^{16 \sigma} +c^*\,k^{c_p}\,e^{4\,\sigma}\,e^{2
a_1\,k^{1/(2-p)}}.
\end{aligned}
\end{equation}

Fixing 
\begin{equation}
\label{ourrealR}
\sigma=d_n\,R^2 =2\,a_1k^{1/(2-p)},
\end{equation}
it follows from \eqref{casi} that, if $\,\|u(0)\|_2\neq 0$, for $k$ large, 
\begin{equation}
\label{uno-1}
\delta(R)\,\geq c_n \|u(0)\|_2\,e^{-M}\,e^{-10 \sigma}=c_n
\|u(0)\|_2\,e^{-M}\,e^{-20\,a_1\,k^{1/(2-p)}},
\end{equation}
for $k\geq k_0(c^*)$ sufficiently large.

We now use our upper bounds for $\delta(R)$ deduced in \eqref{estabove}
\begin{equation}
\label{dos}
\begin{aligned}
&\delta^2(R)= \,\int_{\tfrac{1}{32}}^{\tfrac{31}{32}}\int_{R-1\leq|x|\leq
R}\,(|\widetilde u|^2+\nabla \widetilde u|^2)(x,t)dxdt\\
\\
& \leq  \,\int_{\tfrac{1}{32}}^{\tfrac{31}{32}}\int_{R-1\leq|x|\leq
R}\,e^{\gamma |x|^p} e^{-\gamma |x|^p}\,|\widetilde u|^2(x,t)dxdt\\
\\
&+ c_n \,\int_{\tfrac{1}{32}}^{\tfrac{31}{32}}\int_{R-1\leq|x|\leq R}
t(1-t)\,\frac{(1+|x|)^{2-p}}
{(1+|x|)^{2-p}} \,e^{\gamma |x|^p} e^{-\gamma |x|^p}\,|\nabla \widetilde u|^2\,dxdt\\
\\
\\
&\leq c_n\,e^{-\gamma\,(R-1)^p}\, \sup_{[0,1]}\|e^{\gamma |x|^p/2}
\widetilde u(t)\|_2^2\\\
\\
& + c_n\,\gamma^{-1} \,R^{2-p}\,e^{-\gamma\,(R-1)^p}
\int_{\tfrac{1}{32}}^{\tfrac{31}{32}}\int_{R-1\leq|x|\leq R}
\,\frac{t(1-t)}
{(1+|x|)^{2-p}} e^{\gamma |x|^p} \,|\nabla \widetilde u|^2\,dxdt\\
\\
&\leq c^*\,k^{c_p}\,e^{2a_1 k^{1/(2-p)}}\,e^{-\gamma\,(R-1)^p}.
\end{aligned}
\end{equation}

Gathering the information in \eqref{uno-1}, \eqref{dos}, \eqref{ourrealR}, and
\eqref{constants}  one obtains that
\begin{equation}
\label{popov}
\begin{aligned}
&c_n\, \|u(0)\|^2_2\,e^{-2 M}\leq
c^*\,k^{c_p}\,e^{42\,a_1\,k^{1/(2-p)}-\gamma\,(R-1)^p}
\\
\\
&\leq c^*\,k^{c_p}\,e^{42\,a_1\,k^{1/(2-p)} -
a_0^{1/2}\,(2a_1/d_n)^{p/2}\,k^{1/(2-p)} + O(k^{(1/2(2-p)})}.
\end{aligned}
\end{equation}

Hence, if
\begin{equation}
\label{refer}
42\,a_1<
a_0^{1/2}\,a_1^{p/2}\,(2/d_n)^{p/2},\;\;\,\,\;\;\text{i.e.}\;\;\;\;\;(42)^2\,(d_n/2)^p <
a_0\,a_1^{p-2},
\end{equation}
by letting $k$ tends to infinity it follows from \eqref{popov} that
$\,\|u(0)\|_2\equiv 0$,
which completes our proof.
\end{section}

%%%%%%%%%%%%%%%%%%%%%%%%%%%%%%
\begin{section}{Proof of Corollary  \ref{Corollary 2} and Theorem \ref{Theorem 3}}\label{S: 22}
%%%%%%%%%%%%%%%%%%%%%%

\begin{proof}[Proof of Corollary  \ref{Corollary 2}]
Since
$$
\int\;|u(x,1)|^2\,e^{2\,b\,|x|^q}\,dx\,<
\infty,\;\;\;\;\;\;\text{with}\;\;\;\;b=\beta^q/q,
$$
one has that
$$
\int\;|u(x,1)|^2\,e^{2\,k\,|x|^p}\,dx
\leq
\|e^{2k|x|^p-2b|x|^q}\|_{\infty}\,\int\;|u(x,1)|^2\,e^{2\,b\,|x|^q}\,dx.
$$

A simple calculation shows that
$$
\|e^{2k|x|^p-2b|x|^q}\|_{\infty} = e^{2k|x|^p-2b|x|^q}|_{|x|=r_0},
$$
where
$$
r_0=\left(\frac{pk}{bq}\right)^{1/(q-p)}.
$$
Thus,
$$
\|e^{2k|x|^p-2b|x|^q}\|_{\infty} =  \,e^{2\,a_1 \,k^{q/(q-p)}},
$$
with
\begin{equation}
\label{c1}
a_1=\frac{1}{b^{p/(q-p)}}\,\left[\,\left(\frac{p}{q}\right)^{p/(q-p)}-\left(\frac{p}{q}\right)^{q/(q-p)}\right]
=   c_p\,\frac{1}{b^{p/(q-p)}}.
\end{equation}
Inserting this value in the hypothesis \eqref{conditionp} of Therem 1 we
obtain
$$
a\,\left(\frac{1}{b^{p/(q-p)}}\right)^{p-2} > c_p^{2-p} \,M_p,
$$
with $\,a=\alpha^p/p\,$ and $\,b=\beta^q/q$.  Since  $\;q(2-p)/(q-p)=1$ this gives us 
$$
\alpha\,\beta> N_p,
$$
which yields the result.

\end{proof}
\vskip.1in

\begin{proof}[Proof of Theorem \ref{Theorem 3}]

We just apply Corollary \ref{Corollary 2}  with
$$
u(x,t)=(u_1-u_2)(x,t),
$$
and
$$
V(x,t)= \frac{F(u_1,\overline u_1)- F(u_2,\overline u_2)}{u_1-u_2}.
$$
\end{proof}
\end{section}

%%%%%%%%%%%%%%%%%%%%%%%%%%%%%%
\begin{section}{Proof of Theorem \ref{Theorem 4}}\label{S:  pppp}
%%%%%%%%%%%%%%%%%%%%%%

We shall follow closely the argument used  in the proof of Theorem \ref{Theorem 1}, sections 2-3. Thus, we divide the reasoning  in steps.

First, we deduce the corresponding upper bounds.  Assume $\,\|u(t)\|_2=a\neq 0$. Fix $\,\overline t\in (0,1)$ near $1$, and let
\begin{equation}
\label{4.1}
v(x,t)=u(x,t-1+\overline t),\;\;\;t\in[0,1],
\end{equation}
which satisfies the same equation \eqref{E: NLS1} with
\begin{equation}
\label{4.2}
\begin{aligned}
&|v(x,0)|\leq \frac{b_1}{(2-\overline t)^{n/2}} \,e^{-\frac{b_2 |x|^p}{(2-\overline t)^p}},\\
&|v(x,1)|\leq \frac{b_1}{(1-\overline t)^{n/2}} \,e^{-\frac{b_2|x|^p}{(1-\overline t)^p}}.
\end{aligned}
\end{equation}
So using the notation
\begin{equation}
\label{4.3}
A_0=\frac{b_2}{(2-\overline t)^p},\;\;\;\;\;\;A_1=\frac{b_2}{(1-\overline t)^p},
\end{equation}
one has that
\begin{equation}
\label{4.4}
\begin{aligned}
&\int \,|v(x,0)|^2 \,e^{A_0|x|^p}dx\equiv  a_0^2,\\
&\int \,|v(x,1)|^2 \,e^{A_1|x|^p}dx\equiv a_1^2.
\end{aligned}
\end{equation}

In the remainder of this section $\,c\,$ will denote a constant which may depend on $n,\,b_0,b_1$ and $ b_2$, but is independent of the value of $\,\overline t$.
\begin{equation}
\label{4.5}
V(x,t)=F(u, \overline u),
\end{equation}
so by hypothesis
\begin{equation}
\label{4.6}
|V(x,t)|\leq c\,|u(x,t-1+\overline t)|^{\theta} \leq \frac{c}{(2-t-\overline t)^{\theta n/2}}\,
e^{-\frac{c  |x|^p}{(2-t-\overline t)^p}}.
\end{equation}

We use the conformal or Appell transformation.
If   $v(y,s)$ verifies
\begin{equation}%\label{E: parab&#151;licageneral}
\label{4.6b}
\partial_sv=i\left(\triangle
v+V(y,s)v\right),\;\;\;\;\;\;\;(y,s)\in \R^n\times [0,1],
\end{equation}
and $\alpha$ and $\beta$ are positive,  then
\begin{equation}
\label{4.7}
\widetilde u(x,t)=\left(\tfrac{\sqrt{\alpha\beta}}{\alpha(1-t)+\beta
t}\right)^{\frac n2}v\left(\tfrac{\sqrt{\alpha\beta}\,
x}{\alpha(1-t)+\beta t}, \tfrac{\beta t}{\alpha(1-t)+\beta
t}\right)e^{\frac{\left(\alpha-\beta\right) |x|^2}{4i(\alpha(1-t)+\beta
t)}},
\end{equation}
verifies
\begin{equation}%\label{E: parab&#151;licageneral2}
\label{4.8}
\partial_t\widetilde u=i\left(\triangle \widetilde u+\widetilde
V(x,t)\widetilde u +\widetilde F(x,t)\right),\;\; \text{in}\  \R^n\times
[0,1],
\end{equation}
with
\begin{equation}
\label{4.9}
\widetilde V(x,t)=\tfrac{\alpha\beta}{\left(\alpha(1-t)+\beta
t\right)^2}\,V\left(\tfrac{\sqrt{\alpha\beta}\, x}{\alpha(1-t)+\beta t},
\tfrac{\beta t}{\alpha(1-t)+\beta t}\right),
\end{equation}
\begin{equation}
\label{externalforce2}
\widetilde F(x,t)=\left(\tfrac{\sqrt{\alpha\beta}}{\alpha(1-t)+\beta
t}\right)^{\frac n2+2}F\left(\tfrac{\sqrt{\alpha\beta}\,
x}{\alpha(1-t)+\beta t}, \tfrac{\beta t}{\alpha(1-t)+\beta
t}\right)e^{\frac{\left(\alpha-\beta\right) |x|^2}{4i(\alpha(1-t)+\beta
t)}},
\end{equation}
and
\begin{equation}
\label{4.10}
\begin{aligned}
&\|e^{\gamma |x|^p}\,\widetilde u(x,0)\|_2=\|e^{\gamma(\alpha/\beta)^{p/2}\,|x|^p}\,v(x,0)\|_2=a_0,\\
&\|e^{\gamma |x|^p}\,\widetilde u(x,1)\|_2=\|e^{\gamma(\beta/\alpha)^{p/2}\,|x|^p}\,v(x,1)\|_2=a_1.
\end{aligned}
\end{equation}
We want
$$
2\,\gamma(\alpha/\beta)^{p/2}=A_0,\;\;\;\;2\,\gamma(\beta/\alpha)^{p/2}=A_1,
$$
therefore
\begin{equation}
\label{cconstants}
2\,\gamma=(A_0A_1)^{1/2},\;\;\;\;\;\alpha=A_0^{1/p},\;\;\;\;\;\;\beta=A_1^{1/p}.
\end{equation}
Since
$$
A_0\simeq 1,\;\;\;\;\;\;A_1\simeq \frac{1}{(1-\overline t)^p},
$$
it follows that
\begin{equation}
\label{constants2}
\gamma\simeq \frac{1}{(1-\overline t)^{p/2}},\;\;\;\;\;\;\beta\simeq\frac{1}{(1-\overline t)},\;\;\;\;\;\;\alpha\simeq 1.
\end{equation}

Notice that here the factor $\,1/(1-\overline t)\,$ with $\,\overline t\uparrow 1\;$ plays the role of $\,k\in \Z^+$ with $\,k\uparrow \infty$ in the proof of Theorem 
\ref{Theorem 1} in sections 2 and 3.

Next, we shall estimate  $\|\widetilde V\|_{L^1_tL^{\infty}_{|x|>R}}.$
Thus, 
\begin{equation}
\label{pott1}
|\widetilde V(x,t)|\leq \frac{ \beta}{\alpha}\,|V(y,s)|\leq \frac{\beta}{\alpha}\,\frac{c}{(1-\overline t)^{\theta n/2}}\,e^{- c |y|^p},
\end{equation}
with
$$
|y|= \frac{\sqrt{\alpha \beta} |x|}{\alpha(1-t)+\beta t}\geq \frac {R \sqrt{\alpha}}{\sqrt{\beta}}\simeq \frac{R}{\sqrt{\beta}}=
cR(1-\overline t)^{1/2}.
$$
Thus,
$$
\|\widetilde V\|_{\infty}\leq \frac{\beta}{\alpha}\,\| V\|_{\infty}\leq \frac{c}{1-\overline t}\,\|V\|_{\infty}
\leq \frac{c}{(1-\overline t)^{1+\theta n/2}},
$$
and 
 \begin{equation}
 \label{estt}
 \|\widetilde V(x,t)\|_{L^1_tL^{\infty}(|x|\geq R)} \leq  \| V(y,s)\|_{L^1_sL^{\infty}(|y|\geq cR/\sqrt{\beta})}
 \leq \frac{c}{(1-\overline t)^{\theta n/2}}e^{-  c R^p (1-\overline t)^{p/2}}.
  \end{equation}

To apply Lemma \ref{ultimo} we need 
 $$
  \|\widetilde V(x,t)\|_{L^1_tL^{\infty}(x|\geq R)}\leq \frac{c}{(1-\overline t)^{\theta n/2}}               
 \,e^{-  c R^p (1-\overline t)^{p/2}}\leq \epsilon_0,
 $$
 so  we take in the upper bounds part of the proof
 \begin{equation}
\label{estimateR1}
R^p\simeq \frac{c}{(1-\overline t)^{p/2}}\,log\left(\frac{c}{\epsilon_0\,(1-\overline t)^{\theta n/2}}
\right),
\end{equation}
or
 \begin{equation}
\label{estimateR2}
R\simeq \frac{c}{(1-\overline t)^{1/2}}\,log^{1/p}\left(\frac{c}{\epsilon_0\,(1-\overline t)^{\theta n/2}}
\right).
\end{equation}
 
 Therefore, splitting the term $\,\widetilde V \widetilde u\,$ as in section 2
 \begin{equation}
\label{split2}
\mathbb V=\widetilde V\,\chi_{(|x|>R)}(x),\;\;\;\;\;\;\;\;\;\;\mathbb
F=\widetilde V\,\chi_{(|x|\leq R)}(x)\,\widetilde u(x,t),
\end{equation}
a combination of 
 Lemma \ref{ultimo} and the identity \eqref{est1} yields the estimate
 \begin{equation}
 \label{44.1}
 \begin{aligned}
& \sup_{[0,1]} \|e^{\gamma |x|^p}\tilde u(t)\|_2^2
 \leq c(\|e^{\gamma |x|^p}\tilde u(0)\|_2^2+\|e^{\gamma |x|^p}\tilde u(1)\|_2^2+c 
 a^2\,\|\widetilde V\|^2_{L^{\infty}}\,e^{c\gamma R^p})\\
 \\
 &\leq c\,(a_0^2 +a_1^2+c\,a^2\,\frac{1}{(1-\overline t)^{2+\theta n}}\,e^{cR^p/(1-\overline t)^{p/2}})\\
 \\
&\leq  c\, a^2\,e^{\frac{c}{(1-\overline t)^p}\,log(\frac{c}{\epsilon_0(1-\overline t)^{\theta n/2}})},
\end{aligned}
\end{equation}
\end{section}
where $\,a=\|u_0\|_2$.

We observe that in this case, the upper bound in \eqref{44.1} is coming from the \lq\lq external force" term $\,F=\widetilde V\,\chi_{(|x|\leq R)}(x)\,\widetilde u(x,t)$, and not from the data at time $t=0,\,1$ of $\,\widetilde u\,$ as in the proof of Theorem \ref{Theorem 1}.

Next, using the same argument  given in section 2,  \eqref{771}-\eqref{estabove}, one finds that
\begin{equation}
\label{ww1}
\begin{aligned}
&\gamma\,\int_0^1\int t(1-t)\frac{1}{(1+|x|)^{2-p}}\,|\nabla \widetilde
u(x,t)|^2 e^{\gamma |x|^p}dxdt \\
\\
&\leq  c \,a^2\, e^{\frac{c}{(1-\overline t)^p}\,log(\frac{c}{\epsilon_0(1-\overline t)^{\theta n/2}})}.
 \end{aligned}
 \end{equation}

Now we turn to the lower bounds estimates. Since they  are similar to those given in detail in section 3
we just sketch them. As in \eqref{claim32}
we first need
\begin{equation}
\label{finalbound}
\|\widetilde V\|_{L^{\infty}(\R^n\times [1/32,31/32])}<<R.
\end{equation}

We take in this lower bound part
\begin{equation}
\label{771a}
\sigma=d_nR^2,\;\;\;\;\;\;\;\;\;\;R=\,\frac{c}{(1-\overline t)^{p/2(2-p)}}.
\end{equation}
Since in the time interval $[1/32,31/32]$ one has that $\;\alpha(1-t)+\beta t\simeq \beta$, it follows that
\begin{equation}
\label{770}
\|\widetilde V\|_{L^{\infty}(\R^n\times [1/32,31/32])}\leq \,c\, \frac{ \alpha}{\beta}\;\|V\|_{\infty}
\leq \frac {c}{(1-\overline t)^{\theta n/2-1}}.
\end{equation}

Therefore, \eqref{finalbound} holds if
$$
\frac{\theta n}{2} - 1< \frac{p}{2(2-p)},
$$
or equivalently, 
$$
p>\,\frac{2(\theta n-2)}{\theta n-1},
$$ 
which is exactly our hypothesis, so \eqref{finalbound} holds

Finally, we also need that $\;e^{c\gamma R^p} \,$ to be larger than our upper bound, i.e.
\begin{equation}
\label{uu}
 e^{\frac{c}{(1-\overline t)^p}\,log(\frac{c}{\epsilon_0(1-\overline t)^{\theta n/2}})}
 << e^{c\gamma R^p}=e^{\frac{c}{(1-\overline t)^{p/2}}\,\frac{c}{(1-\overline t)^{p\cdot p/2(2-p)}}}.
 \end{equation}
So it suffices to have
$$
p<p/2+p^2/2(2-p),
$$
which holds if $\,p>1$. This completes  the proof of Theorem \ref{Theorem 4}.

\vskip.1in

\underline{Remark} Let us consider the case $\,\theta=4/n$ in Theorem \ref{Theorem 4}, so that our 
requirement is $\,p>4/3$. The proof of Theorem  \ref{Theorem 4}, in fact, also gives the following result:

\begin{theorem}
\label{Theorem 5}

Assume that $\,u\in C([0,1]:L^2(\mathbb R^n))$ satisfies the equation
$$
i \partial_t u=\Delta u + V(x,t) u,\;\;\;\;\;(x,t)\in \mathbb R^n\times [0,1],
$$
with $V(x,t)$ complex valued, 
$$
|V(x,t)|\leq \frac{c}{(1-t)^2},
$$

\begin{equation}
\label{007}
\lim_{R\to\infty} \|V\|_{L^1([0,1]:L^{\infty}(|x|\geq R))}=0,
\end{equation}

and
$$
|u(x,t)|\leq \frac{c_1}{(1-t)^{n/2}}\,Q\left(\frac{x}{1-t}\right),\;\;\;\;\text{with}\;\;\;\;Q(y)=e^{-c_2|y|^p}.
$$

If $p>4/3$, then $\;u \equiv 0$.

\end{theorem}

It turns out that for complex potentials $V(x,t)$ as in Theorem \ref{Theorem 5}, without the hypothesis \eqref{007}, the restriction $\,p>4/3$  is indeed necessary. This can be seen by performing a pseudo-conformal transformation to the stationary solution furnished by Meshkov's example in \cite{Mes}.

%%%%%%%%%%%%%%%%%%%%%%%%%%%%%%%%%%%%%%%%\end{section}
%%%%%%%%%%%%%%%%%%%%%%%%%%%%%%%%%%%%%
%%%%%%%%%%%%%%%%%%%%%%%%%%%%%%
\begin{section}{Appendix}\label{S:  Algunos Lemmas}
%%%%%%%%%%%%%%%%%%%%%%
%\begin{Appendix}
We shall prove that for any $n\in\Z^+$ and any $p\in(1,2)$ there exists
$c=c(n,p)>1$ such that
for any $x\in \R^n$
$$
\frac{1}{c}\, e^{|x|^p/p}\leq \int_{\R^n}e^{\lambda\cdot
x-|\lambda|^q/q}|\lambda|^{n(q-2)/2}d\lambda\leq c\, e^{|x|^p/p},
$$
i.e.
\begin{equation}
\label{goal}
I_n= \int_{\R^n}e^{\lambda\cdot
x-|\lambda|^q/q}|\lambda|^{n(q-2)/2}d\lambda \sim e^{|x|^p/p}
\end{equation}
with $1/q+1/p=1$.

We shall consider only the case $n\geq 2$ since the case $n=1$ follows
directly from Proposition \ref{stir}.   Assuming that  for any $\mu>1$
\begin{equation}
\label{step1}
\Omega_n=\int_0^{\pi} e^{\mu \cos(\theta)}(\sin(\theta))^{n-2} d\theta\sim
\frac{e^{\mu}}{\mu^{(n-1)/2}},
\end{equation}
we shall prove \eqref{goal}.

Using polar coordinates in $\R^n$ and  \eqref{step1} it follows that
$$
\aligned
I_n&=c_n\int_0^{\infty}\int_0^{\pi}e^{|x|r\cos(\theta)-r^q/q}r^{n(q-2)/2}r^{n-1}(\sin(\theta))^{n-2}d\theta
dr\\
\\
&\sim \int_0^{\infty}\frac{e^{|x|r-r^q/q}}{|x|^{(n-1)/2}}\,
r^{n(q-2)/2}r^{(n-1)/2}dr\\
\\
&=\frac{1}{|x|^{(n-1)/2}}\int_0^{\infty}e^{-|x|(-r+r^q/q|x|)}r^{n(q-1)/2-1/2}dr=\widetilde
I_n.
\endaligned
$$

We recall Stirling's formula (see Proposition 2.1, page 323  in \cite{StSh}).

\begin{proposition}
\label{stir}

If $\Psi$ is a real valued function such that $\Psi'(x_0)=0$ and
$\Psi''(x)>0$ for  $x\in[a,b]$,
and $\Phi $ is a smooth function, then
$$
\int_a^be^{-s\Psi(x) } \Phi(x)dx =
e^{-s\Psi(x_0)}\left[\frac{\sqrt{2\pi}}{s^{1/2}}\frac{\Phi(x_0)}{(\Psi''(x_0))^{1/2}}+O\left(\frac{1}{s}\right)\right],\;\;\;\text{as}\;\;\;s\to\infty.
$$
\end{proposition}

In our case
$$
\aligned
&\Psi(r)=-r+r^q/q|x|,\;\;\;\;\;\;\Psi'(r)=-1+r^{q-1}/|x|=0\;\;\text{if}\;\;r=r_0=|x|^{1/(q-1)},\\
\\
&\Psi''(r)=(q-1)r^{q-2}/|x|,\;\;\;\Psi''(r_0)=(q-1)|x|^{-1/(q-1)},\\
\\
&
\Psi(r_0)=-\frac{1}{p}|x|^{1/(q-1)},
\endaligned
$$
and
$$
\Phi(r)=r^{n(q-1)/2-1/2},\;\;\;\text{so}\;\;\;\Phi(r_0)=|x|^{n/2-1/2(q-1)}.
$$
Therefore,
$$
\widetilde I_n \sim
\frac{1}{|x|^{(n-1)/2}}e^{-|x|\Psi(r_0)}\frac{1}{|x|^{1/2}}\frac{\Phi(r_0)}{(\Psi(r_0))^{1/2}}
\sim
\frac{e^{|x|^p/p}}{|x|^{n/2}}\frac{|x|^{n/2-1/2(q-1)}}{|x|^{-1/2(q-1)}}\sim
e^{|x|^p/p},
$$
which proves \eqref{goal}.

It remains to prove \eqref{step1}. Changing variables and recalling the
fact that $\mu>1$ and $n\geq 2$ we have
$$
\aligned
&\Omega_n=\int_0^{\pi}e^{\mu
\cos(\theta)}(\sin(\theta))^{n-3}\sin(\theta)d\theta=\int_{-1}^1e^{\mu\eta}
(1-\eta^2)^{(n-3)/2}d\eta\\
\\
&\sim \int_{0}^1e^{\mu \eta}(1-\eta)^{(n-3)/2}(1+\eta)^{(n-3)/2}d\eta\sim
\int_{0}^1e^{\mu \eta}(1-\eta)^{(n-3)/2}d\eta\\
\\
&\sim e^{\mu}\int_0^1e^{-\mu \nu} \nu^{(n-3)/2}d\nu
=\frac{e^{\mu}}{\mu^{(n-1)/2}}\int_0^\mu e^{-\rho}\,\rho^{(n-3)/2}\,d\rho\sim \frac{e^{\mu}}{\mu^{(n-1)/2}},
\endaligned
$$
which yields \eqref{step1}.
\end{section}
%\end{Appendix}
%%%%%%%%%%%%

%%%%%%%%%%%%%%%
%%%%%%%%%%%%%%%%%%%%%%%%%%
\end{document}